\documentclass[11pt]{amsproc}
\usepackage{mathrsfs}
\usepackage{amsfonts}
\usepackage{xy}
\usepackage{amsmath,amstext,amsbsy,amssymb}

\newtheorem{theorem}{Theorem}[section]

\newtheorem{proposition}[theorem]{Proposition}

\theoremstyle{definition}

\theoremstyle{remark}

\numberwithin{equation}{section}

\begin{document}
\title{Toroidal Lie superalgebras and free field representations} 
\author{Naihuan Jing}
\address{Department of Mathematics, North Carolina State University,
Raleigh, NC 27695, USA} \email{jing@math.ncsu.edu}
\author{Chongbin Xu*}
\address{School of Sciences, South China University of Technology, Guangzhou 510640
and School of Mathematics \& Information, Wenzhou University,
Zhejiang 325035, China} \email{xuchongbin1977@126.com}
\thanks{Jing gratefully acknowledges the support from
Simons Foundation and CNSF} 

\begin{abstract} A loop-algebraic presentation is given for
toroidal Lie superalgebras of classical types. Based on the loop superalgebra presentation free field realizations
of toroidal Lie superalgebras are constructed for types $A(m,n)$, $B(m,n)$, $C(n)$ and $D(m,n)$.
\end{abstract}
\maketitle

\section{Introduction}
Lie superalgebras and Lie algebras are both important classes of
algebraic structures with ample applications in mathematics and
particularly mathematical physics. Since Kac's classification of
finite-diemsional simple Lie (super)algebra \cite{K1}, there have
been various works on their representations and realizations.

Based on fermioninc realizations \cite{F, KP}, A. Feingold and I. Frenkel \cite{FF} realized
 classical affine Lie algebras using ferminoic fields
 and bosonic fields respectively. Their constructions have been generalized to other algebras such
 as extended affine Lie algebras \cite{G},
 affine Lie superalgebras \cite{KW}, Tits-Kantor-K\"ocher algebras \cite{T}, Lie algebras with
  central extentions \cite{L}, two-parameter quantum affine algebras \cite{JZ} and others.

Toroidal Lie (super)algebras are generalizations of affine
(super)Lie algebras and enjoyed many favorite properties similar to
affine Lie (super)algebras. In the case of 2-toroidal Lie algebras,
the Moody-Rao-Yakonuma presentation shows that the special toroidal
Lie algebras have a similar algebraic structure like the affine Lie
algebras \cite{MRY}, in particular, the double affine
Lie algebras are one subclass. Using the MRY-presentation, the first named author
and collators realized uniformly \cite{JM, JMX} 2-toroidal Lie
algebras of all classical types using bosonic fields or ferminoic
fields with help of a ghost field, which also included the newly discoverd
 bosonic/fermionic
realizations for orthogonal/symplectic types. In the case of super toroidal Lie algebras
we recently constructed a loop-like toroidal Lie superalgebra of type
$B(0,n)$ using bosonic fields and a ghost field in \cite{JX}. That work
suggests that the loop-algebra presentation may be given similarly.

The purpose of this paper is to give a Moody-Rao-Yokonuma presentation
for the toroidal superlagebras of classical types, and then use the presentations
to construct representations of toroidal Lie superalgebras
of type $A,B,C,D$.
Similar like the special case of \cite{JX} we use mixed bosons and fermions
as well as a ghost field.

The structure of this paper is as follows. In section 2, we collect the
preliminaries needed. In section 3, we define a Lie
superalgebra $\mathfrak{T}(X)$ to each type $X$ and give a loop-algebra
presentation for all classical Lie toroidal superalgebras. In section
4, free field representation of $\mathfrak{T}(X)$ is constructed for
each case. In the appendix we list all the extended distinguished Cartan matrices.

\section{Preliminaries}
A Lie superalgebra
$\mathfrak{g}=\mathfrak{g}_{\overline{0}}\oplus\mathfrak{g}_{\overline{1}}$
is  a $\mathbb{Z}_{2}$-graded vector space equipped with a bilinear
map$[\cdot,\cdot]:\mathfrak{g}\times\mathfrak{g}\rightarrow\mathfrak{g}$
such that
\begin{eqnarray*}
  &&\quad1)~
  [\mathfrak{g}_{\alpha},\mathfrak{g}_{\beta}]\subseteq\mathfrak{g}
_{\alpha+\beta}\quad(\mathbb{Z}_{2}-\mbox{gradation}),\\
  &&\quad2) ~[a,b]=-(-1)^{p(a)p(b)}[b,a]\quad(\mbox{graded antisymmetry}),\\
  &&\quad3)~[a,[b,c]]=[[a,b],c]]+(-1)^{p(a)p(b)}[b,[a,c]]\quad(\mbox{graded Jacobi identity}).
  \end{eqnarray*}
 where $\alpha,\beta\in\mathbb{Z}_{2}$ and $a,b,c$ are homogenous elements.

 Let $M,N\in\mathbb{N}$ and $V=V_{\overline{0}}\oplus
V_{\overline{1}}$, where
$\mbox{dim}V_{\overline{0}}=M,\mbox{dim}V_{\overline{1}}=N$. Then
the associative algebra $\mbox{End}V$ is equipped with
$\mathbb{Z}_{2}$-grading
 $\mbox{End}V=\mbox{End}_{\overline{0}} V\oplus\mbox{End}_{\overline{1}}
 V$, where $\mbox{End}_{\alpha} V=\{a\in \mbox{End}V\mid a(V_{s})\subseteq
 V_{s+\alpha}\}$.
 For any two homogenous elements $a,b\in \mbox{End} V $, we define a
 superbracket
$$[a,b]=ab-(-1)^{p(a)p(b)}ba$$
and extended bilinearly, then $\mbox{End}V$ becomes a Lie
superalgebra called general linear superalgebra and denoted by
$\mathfrak{gl}(M|N)$. Let $str$ be the superspace on
$\mathfrak{gl}(M|N)$, then
$\mathfrak{sl}(M,N)=\{a\in\mathfrak{gl}(M|N)|str a=0\}$ is an ideal
of $\mathfrak{gl}(M|N)$ of codimensional 1 called special linear
superalgebras. Note that $\mathfrak{sl}(N,N)$ contains the
1-dimensonal ideal consisting of $\lambda I_{2N}$. We set
\begin{eqnarray*}
  A(m,n) &=& \mathfrak{sl}(m+1,n+1),\quad\mbox{for}~m\neq n,m,n\geqslant0 \\
 A(n,n) &=& \mathfrak{sl}(m+1,n+1)/<I_{2n+2}>\,\quad n>0
\end{eqnarray*}

\subsection{The orthosymplectic
superalgebras} Let $(\cdot|\cdot)$ be a non-degen-erate bilinear
form on $V$ such that $(V_{\overline{0}}|V_{\overline{1}})=0$, the
restriction of $(\cdot|\cdot)$ to $V_{\overline{1}}$ is symmetric
and to $V_{\overline{0}}$ is skewsymmetric, so that $N=2n$ is even.
For $\alpha=0,1$, let
$$
\mathfrak{osp}(M|N)_{\overline{\alpha}}
=\big\{a\in\mathfrak{gl}(M|N)_{\overline{\alpha}}\big|(a(x)|y)
+(-1)^{\alpha p(x)}(x|a(y))=0,x,y\in V \big\}
$$
and $\mathfrak{osp}(M|N)=\mathfrak{osp}(M|N)_{\overline{0}}\oplus
\mathfrak{osp}(M|N)_{\overline{1}}$,
 then
$\mathfrak{osp}(M|N)$ becomes a simple subsuperalgebra of
$\mathfrak{gl}(M|N)$ which is called the orthosymplectic
superalgebras. We denote by respectively
\begin{eqnarray*}
&&B(m,n)=\mathfrak{osp}(2m+1|2n),\quad m\geqslant0,n\geqslant1;\\
&&C(n)=\mathfrak{osp}(2|2n),\quad n\geqslant1;\\
&&D(m,n)=\mathfrak{osp}(2m|2n),\quad m\geqslant0,n\geqslant1
\end{eqnarray*}

Let $(\cdot|\cdot)$ be an even symmetric invariant bilinear form on
linear superalgebra $\mathfrak{g}$, then the associated affine
 superalgebra $\mathfrak{g}^{(1)}$ is
$$\mathfrak{g}^{(1)}=\mathfrak{g}\otimes\mathbb{C}[t,t^{-1}]\oplus \mathbb{C}K\oplus \mathbb{C}d$$
with the following commutation relations
\begin{eqnarray*}
&&[a(m),b(n)]=[a,b](m+n)+m\delta_{m,-n}(a|b)K,\quad[d,a(m)]=-ma(m)
\end{eqnarray*}
where $K$ is central and $a(m)=a\otimes t^{m}; a,b\in
\mathfrak{g};m,n\in \mathbb{Z}$.

 In the framework of Lie superalgebras a special theory
is developed for what one calls the Kac-Moody superalgebra which is
the corresponding theory to that of Kac-Moody algebra and we have the following well-known result.

\begin{proposition} \label{P:KMalg} The linear superalgebras and their
associated affine superalgeras are both Kac-Moody
superalgebras.
\end{proposition}

The following propsotion plays crucial role in defining the
loop-like toroidal linear superalgebras and the equalities is
usually callled the Serre type relations of Kac-Moody
superalgebras (cf. \cite{Van}).

\begin{proposition} Let $\mathfrak{g}(A,\tau)$ be a  Kac-Moody
superalgebra with Chavalley generators $e_{i},f_{i}(1\leqslant
i\leqslant n)$, then the following relations hold for all $i,j$:

If $a_{ii}=a_{ij}=0$, then
$\mbox{ad}e_{i}(e_{j})=\mbox{ad}f_{i}(f_{j})=0$;

If $a_{ii}=0,a_{ij}\neq0$, then
$(\mbox{ad}e_{i})^{2}(e_{j})=(\mbox{ad}f_{i})^{2}(f_{j})=0$;

If $a_{ii}\neq0$, then for $i\neq j$,
$(\mbox{ad}e_{i})^{1-\frac{2a_{ij}}{a_{ii}}}(e_{j})=(\mbox{ad}f_{i})^{1-\frac{2a_{ij}}{a_{ii}}}(f_{j})=0$.
\end{proposition}

Among the simple root systems of linear superalgebra $\mathfrak{g}$,
there exists a simple root system which the numble of odd roots is
the smallest. Such a simple root system is called the distinguished
simple root system and the associated Cartan matrix is called
distinguished Cartan matrix. Furthermore from the distinguished
simple (co)root system, Cartan matrix and Chevalley generators of
$\mathfrak{g}$, one can obtain those
 of $\mathfrak{g}^{(1)}$ in an almost same way adopted in
 the Lie algebra setting. For
convenience, we list the extended distigusihed Cartan matrix of
$\mathfrak{g}$ in the appendix.

 Fix a Vector space $V=V_{\overline{0}}\oplus
V_{\overline{1}}$. A formal distribution $a(z)=\sum_{n\in
\mathbb{Z}}a(n)z^{-n-1}$, where $a(n)\in\mbox{End}V$, is called a
field if for any $v\in V$ one has $a(n)v=0$ for $n\gg0$. Given a
field $a(z)$, we always assume that all the coefficients have the
same parity denoted by $p(a)$ and let
$$a(z)_{-}=\sum_{n\geqslant 0}a(n)z^{-n-1},\quad a(z)_{+}=\sum_{n< 0}a(n)z^{-n-1}$$

For any two fields  $a(z),b(z)$, we define their normally ordered
product as follows:
$$:a(z)b(w):=a(z)_{+}b(w)+(-1)^{p(a)p(b)}b(w)a(z)_{-}$$
It follows from this definition that normally ordered product
satisfies the supercommutativity, i.e.
$:a(z)b(w):=(-1)^{p(a)p(b)}:b(w)a(z):$and $:a(z)b(w):$ is also a
field with the parity of $p(a)+p(b)$. The contraction of any two
fields $a(z),b(w)$ is defined to be:
$$\underbrace{a(z)b(w)}=a(z)b(w)-:a(z)b(w):$$
Furthermore, one can define the normally ordered product of more
than two fields inductively from right to left.

The following well-known Wick's theorem \cite{FLM} is extremely useful for
calculating the OPE of two normally ordered product of free
fields (cf. \cite{K2}).

\begin{theorem} Let $x^{1},x^{2},\cdots,x^{M}$ and
$y^{1},y^{2},\cdots,y^{N}$ be two collections of fields with
definite parity and the notation of normally ordering. Suppose these
fields satisfy the following properties hold:
\begin{eqnarray*}
  1) && [\underbrace{x^{i}y^{j}},z^{k}]=0,\mbox{for all}~ i,j,k ~\mbox{and}~z=x~or~y;\\
  2) &&[x^{i}_{\pm},y^{j}_{\pm}]=0,\mbox{for all}~ i,j.
\end{eqnarray*}
hereafter $[\cdot,\cdot]$ means the superbracket. then we have that
\begin{eqnarray*}
&& :x^{1}x^{2}\cdots x^{M}::y^{1}y^{2}\cdots y^{N}: \\=&&
\sum_{s=0}^{\min(M.N)}\sum_{i_{1}<\cdots<i_{s}\atop
j_{1}\neq\cdots\neq j_{s}}
\Big(\pm\underbrace{x^{i_{1}}y^{j_{1}}}\cdots
\underbrace{x^{i_{s}}y^{j_{s}}}:x^{1}x^{2}\cdots
x^{M}y^{1}y^{2}\cdots
y^{N}:_{(i_{1},\cdots,i_{s},j_{1},\cdots,j_{s})}\Big)
\end{eqnarray*}
where the subscript $(i_{1},\cdots,i_{s},j_{1},\cdots,j_{s})$ means
the fields $x^{i_{1}},\cdots,x^{i_{s}},y^{j_{1}},\cdots,y^{j_{s}}$
are removed and the sign $\pm$ is obtained by the rule: each
permutation of the adjacent odd fields changes the sign.
\end{theorem}

In the final preparation of formal calculcus, we recall the
definition of the formal delta-function:
$\delta(z-w)=\sum_{n\in\mathbb{Z}}z^{-n-1}w^{n}$
which is formally defined to be the series expansions
in two directions:
 $$\partial^{(j)}_{w}\delta(z-w)=i_{z,w}\frac{1}{(z-w)^{j+1}}-i_{w,z}\frac{1}{(z-w)^{j+1}}$$
 where $\partial^{(j)}_{w}=\partial^{j)}_{w}/j!$ and the symbol $i_{z,w}$(resp.
 $i_{w,z}$) means power series expansion in the domain
 $|z|>|w|$ (resp.  $|z|>|w|$). In additional, the  equality holds when both
 sides make sense :$ f(z,w)\delta(z-w)=f(z,z)\delta(z-w)$. Usually we will
 drop $i_{z, w}$ if it is clear from the context.

\section{Loop-like toroidal  Lie superalgebras $\mathfrak{T}(X)$}

Let $R=\mathbb C[s^{\pm1},t^{\pm1}]$ be the complex commutative ring
of Laurant polynomials in two variables $s, t$. Let $\mathfrak g$ be
a complex Lie superalgebra. The loop Lie superalgebra $L(\mathfrak
g):=\mathfrak g\otimes R$ is defined under the Lie
superbracket $[x\otimes a, y\otimes b]=[x, y]\otimes ab$. Let
$\Omega_R$ be the $R$-module of K\"ahler differentials of $R$
spanned by $da, a\in R$, and let $d\Omega_R$ be the space of exact
forms. The quotient $\Omega_R/d\Omega_R$ has a basis consisting of
$\overline{s^{m-1}t^{n}ds}$, $\overline{s^{n}t^{-1}dt}$,
$\overline{s^{-1}ds}$, where $m, n\in\mathbb Z$. Here $\overline{a}$
denotes the coset $a+d\Omega_R$.

Let $T(\mathfrak g)$ be the
{\it toroidal superalgebra} Lie superalgebra:
\begin{equation*}
T({\mathfrak g})={\mathfrak g}\otimes R\oplus \Omega_R/d\Omega_R
\end{equation*}
with the Lie superbrackets ($x,y\in \mathfrak{g},~a,b\in R$)
$$[x\otimes a, y\otimes b]=[x, y]\otimes ab+(x|y)\overline{(da)b},
\quad [T({\mathfrak g}), \Omega_R/d\Omega_R]=0$$
and the parities are specified by: $$p(x\otimes a)=p(x),\quad
p(\Omega_R/d\Omega_R)=\overline{0}$$

\begin{theorem}\cite{IK} The {\it toroidal Lie superalgebra} $T(\mathfrak{g})$ is the
universal central extension of the loop Lie superalgebra $\mathfrak
{g}\otimes R$.
\end{theorem}

Here a universal central extension is assumed to be perfect, thus there
is a unique universal central extension as in the Lie algebra cases
\cite{Ga, W}.
Next we would like to give a loop algebra presentation for the
toroidal Lie superalgebras of classical types as the Moody-Rao-Yokonuma
presentation for toroidal Lie algebras \cite{MRY}.

Let $\mathfrak g$ be the Lie superalgebra associated with the
  extended distinguished Cartan matrix $A=(a_{ij})$ of type $X^{(1)}$. Let
$Q=\mathbb{Z}\alpha_{0}\oplus\cdots\oplus \mathbb{Z}\alpha_{r}$ be
its root lattice, where $r=m+n+1; m+n; n+1$ respectively for
$X=A(m,n)$; $B(m,n),D(m,n)$; $X=C(n)$. The odd simple roots are:
$$\Pi\cap\Delta_{\overline{1}}
=\left\{
\begin{array}{ll}
\{\alpha_{0},\alpha_{n}\}& ~ \mbox{if}~X=A(m,n);\\
\{\alpha_{n}\}& ~ \mbox{if}~X=B(m,n),~D(m,n);\\
\{\alpha_{0},\alpha_{1}\}&~ \mbox{if}~X=C(n)
\end{array} \right.$$
The standard invariant form is given by
$(\alpha_{i},\alpha_{j})=d_{i}a_{ij}$, where
$$(d_{0},d_{1},\cdots,d_{r})
=\left\{
\begin{array}{lllll}
(1,\underbrace{1,\cdots,1}_{m+1},\underbrace{-1,\cdots,-1}_{n}),&\mbox{if}~X=A(m,n);\\
(2,\underbrace{1,\cdots,1}_{n-1},1/2),&\mbox{if}~X=B(0,n);\\
(2,\underbrace{1,\cdots,1}_{n},\underbrace{-1,\cdots,-1}_{m-1},-1/2),& \mbox{if}~ X=B(m,n), m\geqslant1;\\
(1,1,\underbrace{-1,\cdots,-1}_{n-1},-2),& \mbox{if}~ X=C(n);\\
(2,\underbrace{1,\cdots,1}_{n},\underbrace{-1,\cdots,-1}_{m}),& \mbox{if}~ X=D(m,n).\\
\end{array} \right.$$
Note that $d_i=\frac12(\alpha_i, \alpha_i)$ for non-isotropic roots.

\begin{theorem} \label{eq:presentation}
The  toroidal Lie superalgebra $T(\mathfrak g)$ 
is isomorphic to the Lie superalgebra  generated by
$$\{\mathcal{K},\alpha_{i}(k),x^{\pm}_{i}(k)|\, 0\leqslant i\leqslant
r,k\in\mathbb{Z}\}$$
 with parities given as : ($0\leqslant i\leqslant r,k\in\mathbb{Z}$)
$$p(\mathcal{K})=p(\alpha_{i}(k))=\overline{0},\quad p(x^{\pm}_{i}(k))=p(\alpha_{i})$$
subject to the following relations:
\begin{eqnarray*}
   &1)&[\mathcal{K},\alpha_{i}(k)]=[\mathcal{K},x^{\pm}_{i}(k)]=0; \\
   &2)& [\alpha_{i}(k),\alpha_{j}(l)]=k(\alpha_{i}|\alpha_{j})\delta_{m,-n}\mathcal{K}; \\
   &3)& [\alpha_{i}(k),x^{\pm}_{j}(l)]=\pm(\alpha_{i}|\alpha_{j})x^{\pm}_{j}(k+l); \\
   &4)& [x^{+}_{i}(k),x^{-}_{j}(l)]=0,\mbox{if}~i\neq j;\\
   &&[ x^{+}_{i}(k),x^{-}_{i}(l)]=-\frac{2}{(\alpha_{i}|\alpha_{i})}\{\alpha_{i}(k+l)
   +k\delta_{k,-l}\mathcal{K}\}, \mbox{if}~(\alpha_{i}|\alpha_{i})\neq0;\\
   &5)&[x^{\pm}_{i}(k),x^{\pm}_{i}(l)]=0;\\
   &&[x^{\pm}_{i}(k),x^{\pm}_{j}(l)]=0,\mbox{if}~a_{ii}=a_{ij}=0,i\neq j;\\
&&[x^{\pm}_{i}(k),[x^{\pm}_{i}(k),(x^{\pm}_{j}(l)]]=0,\mbox{if}~a_{ii}=0,a_{ij}\neq0;\\
   &&\underbrace{[x^{\pm}_{i}(k),\cdots,[}_{1-a_{ij}}x^{\pm}_{i}(k),x^{\pm}_{j}(l)]\cdots]=0,\mbox{if}~a_{ii}\neq0,i\neq
   j.
\end{eqnarray*}
\end{theorem}
Let $\mathfrak{T}(X)$ be the algebra generated by
$\{\mathcal{K},\alpha_{i}(k),x^{\pm}_{i}(k)|\, 0\leqslant i\leqslant
r,k\in\mathbb{Z}\}$ defined in Theorem \ref{eq:presentation}.
The algebra $\mathfrak{T}(X)$
is a $Q\times \mathbb Z$-graded Lie
superalgebra under the grading: $\mbox{deg} \mathcal{K}=(0, 0),
\mbox{deg} \alpha_i(k)=(0, k), \mbox{deg} x^{\pm}_i(k)=(\pm\alpha_i,
k)$. We denote the subspace of degree $(\alpha, k)$ by $\mathfrak
T_k^{\alpha}$, then $\displaystyle \mathfrak T(X)=\oplus_{(\alpha,
k)\in Q\times \mathbb Z}\mathfrak T_k^{\alpha}(X)$. We remark that
the center of this algebra is contained in the subalgebra generated
by $\mathfrak T_k^{n\delta}(X)$, $k, n\in \mathbb Z$.

\begin{proposition} \label{homo} The following map defines a surjective
homomorphism from loop-like toroidal superalgebra $\mathfrak T(X)$
to the algebra $T(X)$:
\begin{align*}
\mathcal{K}&\mapsto \overline{s^{-1}ds}\\
\alpha_i(k)&\mapsto d_i(h_i\otimes s^k+\delta_{i0}\overline{s^kt^{-1}dt}), \qquad i=0, \cdots, n\\
x^+_i(k)&\mapsto e_i\otimes s^k, \qquad i=1, \cdots, n\\
x^-_i(k)&\mapsto -f_i\otimes s^k, \qquad i=1, \cdots, n\\
x^+_0(k)&\mapsto e_0\otimes s^kt^{-1}, \\
x^-_0(k)&\mapsto -f_0\otimes s^kt
\end{align*}
\end{proposition}
\begin{proof} It is straightforward to check the elements on the
right hand satisfy the relation 1)-4). Note that $a_{ii}=2$ if
$a_{ii}\neq 0$, then the last relation is the direct result of
Proposition \ref{P:KMalg}.
\end{proof}

{\it Proof of Theorem \ref{eq:presentation}}. It follows from
Proposition \ref{homo} that the loop-like toroidal Lie superalgebra
$\mathfrak T(X)$ is a central extension of the 2-loop superalgebra
$\mathfrak g\otimes \mathbb C[t, t^{-1}, s, s^{-1}]$, as the kernel
is contained in the subspace $\oplus_{n,k} \mathfrak T^{n\delta}_k$,
which is clearly central by the commutation relations. On the other
hand it is straightforward to check that the algebra $\mathfrak
T(X)$ is a central extension of the loop algebra $L(\mathfrak g)$
and therefore the central extension is isomorphic to $\mathfrak
T(x)$ by the similar arguments as in \cite{MRY}. \qed.

 It will be convenient to rewrite the relations in terms of generating series.
\begin{proposition}\emph{The relations of $\mathfrak{T}(X)$ can be written as
follows.}
 \begin{eqnarray*}
    &1')& [\mathcal{K},\alpha_{i}(z)]=[\mathcal{K},x^{\pm}_{i}(z)]=0; \\
    &2')& [\alpha_{i}(z),\alpha_{j}(w)]=(\alpha_{i}|\alpha_{j})\partial_{w}\delta(z-w)\mathcal{K};\\
    &3')& [\alpha_{i}(z),x^{\pm}_{j}(w)]=\pm(\alpha_{i}|\alpha_{j})x^{\pm}_{j}(w)\delta(z-w); \\
    &4')& [x^{+}_{i}(z),x^{-}_{j}(w)]=0,\mbox{\emph{if}}~i\neq j   ;\\
    &&[x^{+}_{i}(z),x^{-}_{i}(w)]=-\{(\alpha_{i}(w)
    \delta(z-w)+\partial_{w}\delta(z-w)\mathcal{K}\},\mbox{\emph{if}}~(\alpha_{i}|\alpha_{i})=0\\
    &&[x^{+}_{i}(z),x^{-}_{i}(w)]=-\frac{2}{(\alpha_{i}|\alpha_{i})}\{(\alpha_{i}(w)
    \delta(z-w)+\partial_{w}\delta(z-w)\mathcal{K}\},\mbox{\emph{if}}~(\alpha_{i}|\alpha_{i})\neq0\\
    &5')&[x^{\pm}_{i}(z),x^{\pm}_{i}(w)]=0;\\
    &&[x^{\pm}_{i}(z),x^{\pm}_{j}(w)]=0,\mbox{\emph{if}}~a_{ii}=a_{ij}=0,i\neq j;\\
    &&[x^{\pm}_{i}(z_{1}),[x^{\pm}_{i}(z_{2}),x^{\pm}_{j}(w)]]=0,\mbox{\emph{if}}~a_{ii}=0,a_{ij}\neq0;\\
    &&[x^{\pm}_{i}(z_{1}),\cdots,[x^{\pm}_{i}(z_{1-a_{ij}}),x^{\pm}_{j}(w)]\cdots]=0,
    \mbox{\emph{if}}~a_{ii}\neq0,i\neq j.
    \end{eqnarray*}
where
$\alpha_{i}(z)=\sum_{k\in\mathbb{Z}}\alpha_{i}(k)z^{-k-1},\quad
x^{\pm}_{i}(z) =\sum_{k\in\mathbb{Z}}x^{\pm}_{i}(k)z^{-k-1}.$
\end{proposition}
\section{Free field realizations of $\mathfrak{T}(X)$}
In this section, we shall construct free field representations of
toroidal superalgebras defined in last section.

 For type $A(m,n)$, let $\varepsilon_{i}(0\leqslant i\leqslant n+m+3)$ be an orthomormal
basis of the vector space $\mathbb{C}^{n+m+4}$ and
$\delta_{i}=\sqrt{-1}\varepsilon_{m+1+i}~(1\leqslant i\leqslant
n+1)$. For type $B(m,n),D(m,n)$£¬ let $\varepsilon_{i}(0\leqslant
i\leqslant n+m+1)$ be an orthomormal basis of the vector space
$\mathbb{C}^{n+m+2}$ and
$\delta_{i}=\sqrt{-1}\varepsilon_{n+i}~(1\leqslant i \leqslant m+1)$. 
For type $C(n)$, let $\varepsilon_{i}(0\leqslant i\leqslant
n+m+2)$ be an orthomormal basis of the vector space
$\mathbb{C}^{n+3}$ and denote by
$\delta_{i}=\sqrt{-1}\varepsilon_{1+i}~(1\leqslant i\leqslant n+1)$.
Then the distinguished simple root systems, positive root systems
and longest distinguished root of the orthosymplectic superalgebras
can be expressed in terms of vectors $\varepsilon_{i}$'s and
$\delta_{i}$ 's as follows:
\begin{eqnarray*}
&\mbox{i})&\mbox{Type}~A(m,n)~(m,n\geqslant1):\\
&&\Pi=\big\{\alpha_{1}=\varepsilon_{1}-\varepsilon_{2},\cdots,
,\alpha_{m}=\varepsilon_{m}-\varepsilon_{m+1},\alpha_{m+1}=\varepsilon_{m+1}-\delta_{1},\\
&&\qquad\alpha_{n+2}=\delta_{1}-\delta_{2},\cdots,\alpha_{n+m+1}=\delta_{n}-\delta_{n+1}\big\};\\
&&\triangle_{+}=\big\{\varepsilon_{i}-\varepsilon_{j},\delta_{k}-\delta_{l}|1\leqslant
i<j\leqslant n+1,1\leqslant k<l\leqslant m+1\big\}\\
&&\qquad\cup\big\{\delta_{k}-\varepsilon_{i}|1\leqslant i\leqslant n+1,1\leqslant k\leqslant m+1\big\};\\
&&\theta=\alpha_{1}+\cdots+\alpha_{m+n+1}=\varepsilon_{1}-\delta_{n+1}.\\
&\mbox{ii})&\mbox{Type}~B(0,n)~(n\geqslant1):\\
&&\Pi=\big\{\alpha_{1}=\varepsilon_{1}-\varepsilon_{2},\cdots,
\alpha_{n-1}=\varepsilon_{n-1}-\varepsilon_{n},\alpha_{n}=\varepsilon_{n}\big\};\\
&&\triangle_{+}=\big\{\varepsilon_{i}\pm\varepsilon_{j}|1\leqslant i<j\leqslant n\big\}
\cup\big\{\varepsilon_{i},2\varepsilon_{i}|1\leqslant i\leqslant n\big\};\\
&&\theta=2\alpha_{1}+\cdots+2\alpha_{n}=2\varepsilon_{1}.\\
&\mbox{iii})&\mbox{Type}B(m,n)~(m\geqslant1,n\geqslant1):\\
&&\Pi=\big\{\alpha_{1}=\varepsilon_{1}-\varepsilon_{2},\cdots,
\alpha_{n-1}=\varepsilon_{n-1}-\varepsilon_{n},\alpha_{n}=\varepsilon_{n}-\delta_{1}\\
&&\quad\qquad\alpha_{n+1}=\delta_{1}-\delta_{2},\cdots,\alpha_{n+m-1}=\delta_{m-1}-\delta_{m},
\alpha_{n+m}=\delta_{m}\big\}\\
&&\triangle_{+}=\big\{\varepsilon_{i}\pm\varepsilon_{j},
\delta_{k}\pm\delta_{l},|1\leqslant i< j\leqslant n,1\leqslant k<
l\leqslant m\big\}\\
&&\quad\quad\cup
\big\{2\varepsilon_{i},\delta_{k},\varepsilon_{i}\pm\delta_{k}|
1\leqslant i\leqslant n,1\leqslant k\leqslant m\big\}\\
&&\theta=2\alpha_{1}+\cdots+2\alpha_{n+m}=2\varepsilon_{1}.\\
&\mbox{iv})& \mbox{Type}~D(m,n)~(m\geqslant2,n\geqslant1):\\
&&\triangle_{+}=\big\{\varepsilon_{i}\pm\varepsilon_{j},
\delta_{k}\pm\delta_{l}|1\leqslant i<j\leqslant n,1\leqslant k< l\leqslant m\big\},\\
&&\qquad\cup\big\{2\delta_{k},\varepsilon_{i}\pm\delta_{i}|1\leqslant
i\leqslant n,1\leqslant k\leqslant n\big\}\\
&&\Pi=\big\{\alpha_{1}=\varepsilon_{1}-\varepsilon_{2},\cdots,
\alpha_{n-1}=\varepsilon_{n-1}-\varepsilon_{n},\alpha_{n}=\varepsilon_{n}-\delta_{1},\\
&&\qquad\cdots\alpha_{n+1}=\delta_{1}-\delta_{2},
\alpha_{n+m+1}=\delta_{m-1}-\varepsilon_{m},\alpha_{n+m}=\delta_{m-1}+\delta{m}\big\}\\
&&\theta=2\alpha_{1}+\cdots+2\alpha_{n+m-2}+\alpha_{n+m-1}+\alpha_{n+m}=2\varepsilon_{1}\\
&\mbox{v})&\mbox{Type}~C(n)~(n\geqslant1):\\
&&\Pi=\big\{\alpha_{1}=\varepsilon_{1}-\delta_{1},\alpha_{2}=\delta_{1}-\delta_{2},\cdots,
\alpha_{n}=\delta_{n-1}-\delta_{n},\alpha_{n+1}=2\delta_{n}\big\};\\
&&\triangle_{+}=\big\{\delta_{k}\pm\delta_{l}|1\leqslant
k<l\leqslant n\big\}\cup
\big\{2\delta_{k},\varepsilon_{1}\pm\delta_{k}|1\leqslant k\leqslant
n\big\};\\
&&\theta=\alpha_{1}+2\alpha_{2}+\cdots+2\alpha_{n}+\alpha_{n+1}=\varepsilon_{1}+\delta_{1}.\\
\end{eqnarray*}

We introduce  $\overline{c}=\varepsilon_{0}+\delta_{n+2}$ for type
$A(m,n)$, $\overline{c}=\varepsilon_{0}+\delta_{m+1}$ for type
$B(m,n),D(m,n)$ and $\overline{c}=\varepsilon_{0}+\delta_{n+1}$ for
type $C(n)$, then $\alpha_{0}=\overline{c}-\theta$. Furthermore we
define
$$\beta=\left\{
\begin{array}{lll}
\varepsilon_{1}-\overline{c},&\mbox{for}~A(m,n);\\
\varepsilon_{1}-\frac{1}{2}\overline{c},&\mbox{for}~B(m,n),D(m,n);\\
\delta_{1}-\overline{c},&\mbox{for}~C(n).
\end{array}
\right.$$ Then we have
$(\beta|\beta)=1,(\beta|\varepsilon_{i})=\delta_{1i}$,
$\alpha_{0}=-\beta+\delta_{n+1}$ for type $A(m,n)$
$(\beta|\beta)=1,(\beta|\varepsilon_{i})=\delta_{1i},
\alpha_{0}=-2\beta$ for type $B(m,n),D(m,n)$ and
$(\beta|\beta)=-1,(\beta|\delta_{i})=-\delta_{1i},\alpha_{0}=-\beta-\varepsilon_{1}$
for type $C(n)$.

\bigskip

Let $\mathcal {P}_{\overline{0}}$ and $\mathcal {P}_{\overline{1}}$
be the vector spaces defined  for each case  as follows:
$$\mathcal {P}_{\overline{0}}=\left\{
 \begin{array}{lll}
\mbox{span}_{\mathbb{C}}\{\overline{c},\varepsilon_{i}|1\leqslant
i\leqslant m+1\},&\quad\mbox{for}~A(m,n)\\
\mbox{span}_{\mathbb{C}}\{\overline{c},\varepsilon_{i}|1\leqslant
i\leqslant n\},&\quad\mbox{for}~B(m,n),D(m,n)\\
\mbox{span}_{\mathbb{C}}\{\overline{c},\delta_{i}|1\leqslant
i\leqslant n\},&\quad\mbox{for}~C(n).
 \end{array}
 \right.$$
 $$\mathcal {P}_{\overline{1}}=\left\{
 \begin{array}{lll}
\mbox{span}_{\mathbb{C}}\{\delta_{i}|1\leqslant
i\leqslant m+1\},&\quad\mbox{for}~A(m,n)\\
\mbox{span}_{\mathbb{C}}\{\delta_{i}|1\leqslant
i\leqslant n\},&\quad\mbox{for}~B(m,n),D(m,n)\\
\mbox{span}_{\mathbb{C}}\{\varepsilon_{1}\},&\quad\mbox{for}~C(n).
 \end{array}
 \right.$$
and set $\mathcal {C}_{\overline{0}}=\mathcal
{P}_{\overline{0}}\oplus\mathcal {P}_{\overline{0}}^{*},\mathcal
{C}_{\overline{1}}=\mathcal
{P}_{\overline{1}}\oplus\mathcal{P}_{\overline{1}}^{*}$, then we
form the superspace: $\mathcal {C}=\mathcal
{C}_{\overline{0}}\oplus\mathcal {C}_{\overline{1}}$. We define a
bilinear form on $\mathcal {C}$ by: $\mbox{for}~ a,b\in\mathcal
{P}_{\overline{0}}\cup\mathcal {P}_{\overline{1}}$
\begin{eqnarray*}
   && \langle b^{*},a\rangle=-(-1)^{p(a)p(b)}\langle
a,b^{*}\rangle=(a,b); \\
  &&\langle b,a\rangle=-(-1)^{p(a)p(b)}\langle
a^{*},b^{*}\rangle=0,
\end{eqnarray*}

Let $\mathcal {A}(\mathbb{Z}^{2n+2|2m})$ be the associated
superalgebra generated by
$$\{u(k)|u\in \mathcal {C}_{\overline{0}}\cup \mathcal {C}_{\overline{1}},k\in\mathbb{Z}\}$$
with the parity $p(u(k))=p(u)$ and the defining relations
$$u(k)v(l)-(-1)^{p(k)p(l)}u(k)v(l)=\langle u,v\rangle\delta_{k,-l}$$
for $u,v\in \mathcal {C}_{\overline{0}}\cup\mathcal
{C}_{\overline{1}}$ and $ k,l\in\mathbb{Z}$.

The representation space of the superalgebras $\mathcal
{A}(\mathbb{Z}^{2n+1|2m})$ is defined to be the following vector
space:

$$V(\mathbb{Z}^{n+1|m})=\bigotimes_{a_{i}}\Big(\bigotimes_{k\in\mathbb{Z}_{+}}\mathbb{C}[a_{i}(-k)]
\bigotimes_{k\in\mathbb{Z}_{+}}\mathbb{C}[a^{*}_{i}(-k)]\Big)$$
where $a_{i}$ runs though any basis in $\mathcal{P}_{\overline{0}}$
and $\mathcal{P}_{\overline{1}}$, consisting of, say
$\overline{c},\varepsilon_{i}$'s and $\delta_{k}$'s. The
superalgebra $\mathcal {A}(\mathbb{Z}^{2n+1|2m})$ acts on the space
by the usual action: $a(-k)$ act as creation operators and $a(k)$ as
annihation operators.

To construct $\mathfrak{T}(B(m,n))$, we need to extend the
superalgebra $\mathcal {A}(\mathbb{Z}^{2n+1|2m})$ by adding a new
$e$ to the basis of the odd space $\mathcal {C}_{\overline{1}}$  and
extending the bilinear form define above by $\langle
e,e\rangle=-2,\quad\langle e,\mathcal {C}\rangle=0$. The larger
superalgebra is denoted by  $\mathcal {A}(\mathbb{Z}^{2n+1|2m+1})$
and its representation space is defined to be
$$V(\mathbb{Z}^{n+1|m+1})=V(\mathbb{Z}^{n+1|m})\bigotimes_{k\in\mathbb{Z}_{+}}e(-k)$$
with $e(-k)(\mbox{resp.}~e(k))$ act as creation (resp.annihation)
operators and $e(0)$ as scalar $\sqrt{-1}$.

For any $u\in \mathcal{C}_{0}\cup\mathcal{C}_{1}$, we define the
following fromal power series with cofficients coming from the
superalgebra $\mathcal{A}(\mathbb{Z}^{2n+1|2m})$:
$$u(z)=\sum_{k\in\mathbb{Z}}u(k)z^{-k-1}$$
It is easy to see that $u(z)$ is a field on $V$ with the parity
$p(u)$.

\begin{proposition} The basic operator product expansions are :for
$u,v\in \mathcal{C}_{0}\cup\mathcal{C}_{1}$, we have
\begin{eqnarray*}
  \underbrace{u(z)v(w)}=\frac{\langle u,v\rangle}{z-w}
\end{eqnarray*}
In particular we have for
$a,b\in\mathcal{P}_{\overline{0}}\cup\mathcal{P}_{\overline{1}}$
\begin{eqnarray*}
   &&\underbrace{a(z)b(w)}= \underbrace{a^{*}(z)b^{*}(w)}=0;\\
   &&\underbrace{a(z)b^{*}(w)}=\frac{\langle a,b^{*}\rangle}{z-w},\quad
   \underbrace{a^{*}(z)b(w)}=\frac{\langle a^{*},b\rangle}{z-w}\\
   &&\underbrace{e(z)a(w)}= \underbrace{e(z)a^{*}(w)}=0;\quad
\underbrace{e(z)e(w)}=\frac{-2}{z-w}
\end{eqnarray*}
\end{proposition}
\begin{proof} One can prove the propostion similarly to Proposition 3.1 in \cite{JMX}.
\end{proof}
\begin{proposition}
For $u,v\in \mathcal{C}_{0}\cup\mathcal{C}_{1}$, we have
\begin{eqnarray*}
  &&[u(z),v(w)]=\langle u,v\rangle\delta(z-w),\\
  &&[e(z),u(w)]=0,\quad [e(z),e(w)]=-2\delta(z-w)
\end{eqnarray*}
\end{proposition}
\begin{proof} In fact, we have
\begin{eqnarray*}
   [u(z),v(w)]&=& u(z)v(w)-(-1)^{p(a)p(b)}v(w)u(z) \\
   &=&\frac{\langle u,v\rangle}{z-w}-(-1)^{p(a)p(b)}\frac{\langle v,u\rangle}{w-z} \\
   &=&\langle u,v\rangle\delta(z-w)
\end{eqnarray*}
where we have used the fact $:u(z)v(w):=(-1)^{p(a)p(b)}:v(w)u(z):$
\end{proof}

The following proposition can be proved by the Wick's theorem.

\begin{proposition} \label{P:comm1} The commutators among normally ordered products
obey the following rules: \\
 \emph{1)}If $r_{i}\in
\mathcal{C}_{\overline{0}},1\leqslant i\leqslant 4$, then
\begin{eqnarray*}
   && [:r_{1}(z)r_{2}(z):,:r_{3}(w)r_{4}(w):]\\
  &=& \langle r_{1},r_{4}\rangle :r_{2}(z)r_{3}(z):\delta(z-w)+ \langle r_{1},r_{3}\rangle :r_{2}(z)r_{4}(z):\delta(z-w)\\
  &+& \langle r_{2},r_{3}\rangle :r_{1}(z)r_{4}(z):\delta(z-w)+ \langle r_{2},r_{4}\rangle :r_{1}(z)r_{3}(z):\delta(z-w) \\
   &+&\big(\langle r_{1},r_{4}\rangle\langle r_{2},r_{3}\rangle+
   \langle r_{1},r_{3}\rangle \langle r_{2},r_{4}\rangle\big)\partial_{w}\delta(z-w)
\end{eqnarray*}
\emph{2)}If $r_{1},r_{2}\in \mathcal{C}_{\overline{0}},r_{3}\in
\mathcal{C}_{\overline{1}}$ and $ r_{4}\in
\mathcal{C}_{\overline{1}}\cup\{e\}$, then
\begin{eqnarray*}
   && [:r_{1}(z)r_{2}(z):,:r_{3}(w)r_{4}(w):]=0
\end{eqnarray*}
\emph{3)}If $r_{1},r_{3}\in \mathcal{C}_{\overline{1}}$ and
$r_{2},r_{4}\in \mathcal{C}_{\overline{1}}\cup\{e\}$, then
\begin{eqnarray*}
   && [:r_{1}(z)r_{2}(z):,:r_{3}(w)r_{4}(w):]\\
  &=& \langle r_{1},r_{4}\rangle :r_{2}(z)r_{3}(z):\delta(z-w)- \langle r_{1},r_{3}\rangle :r_{2}(z)r_{4}(z):\delta(z-w)\\
  &+& \langle r_{2},r_{3}\rangle :r_{1}(z)r_{4}(z):\delta(z-w)- \langle r_{2},r_{4}\rangle :r_{1}(z)r_{3}(z):\delta(z-w) \\
   &+&\big(\langle r_{1},r_{4}\rangle\langle r_{2},r_{3}\rangle-
   \langle r_{1},r_{3}\rangle \langle r_{2},r_{4}\rangle\big)\partial_{w}\delta(z-w)
\end{eqnarray*}
\emph{4)}If $r_{1},r_{3}\in \mathcal{C}_{\overline{0}}$ and
$r_{2},r_{4}\in \mathcal{C}_{\overline{1}}\cup\{e\}$, then
\begin{eqnarray*}
   && [:r_{1}(z)r_{2}(z):,:r_{3}(w)r_{4}(w):]\\
  &=& \langle r_{1},r_{3}\rangle :r_{2}(z)r_{4}(z):\delta(z-w)+
   \langle r_{2},r_{4}\rangle :r_{1}(z)r_{3}(z):\delta(z-w) \\
   &+&\langle r_{1},r_{3}\rangle\langle r_{2},r_{4}\rangle\partial_{w}\delta(z-w)
\end{eqnarray*}
\emph{5)}If $r_{1},r_{2},r_{3}\in \mathcal{C}_{\overline{0}}$ and
$r_{4}\in \mathcal{C}_{\overline{1}}\cup\{e\}$, then
\begin{eqnarray*}
   && [:r_{1}(z)r_{2}(z):,:r_{3}(w)r_{4}(w):]\\
  &=& \langle r_{1},r_{3}\rangle :r_{2}(z)r_{4}(z):\delta(z-w)+
   \langle r_{2},r_{3}\rangle :r_{1}(z)r_{4}(z):\delta(z-w)
\end{eqnarray*}
\emph{6)}If $r_{1},r_{2}\in \mathcal{C}_{\overline{1}}$ and
$r_{3}\in\mathcal{C}_{0},r_{4}\in
\mathcal{C}_{\overline{1}}\cup\{e\}$, then
\begin{eqnarray*}
   && [:r_{1}(z)r_{2}(z):,:r_{3}(w)r_{4}(w):]\\
  &=& -\langle r_{1},r_{4}\rangle :r_{2}(z)r_{3}(z):\delta(z-w)+ \langle r_{1},r_{3}\rangle :r_{2}(z)r_{4}(z):\delta(z-w)\\
  && -\langle r_{2},r_{3}\rangle :r_{1}(z)r_{4}(z):\delta(z-w)+ \langle r_{2},r_{4}\rangle :r_{1}(z)r_{3}(z):\delta(z-w) \\
   &&+\big(\langle r_{1},r_{3}\rangle \langle r_{2},r_{4}\rangle-\langle r_{1},r_{4}\rangle\langle r_{2},r_{3}\rangle\big)\partial_{w}\delta(z-w)
\end{eqnarray*}
\end{proposition}

In the following we will define the field operators to each
generating series of $\mathfrak{T}(X)$.

(i) For type $A(m,n)$, we define
$$x_{i}^{+}(z)=
\left\{\begin{array}{lll}
\sqrt{-1}:\delta_{n+1}(z)\beta^{*}(z):,&\quad\mbox{if}~ i=0;\\
\sqrt{-1}:\varepsilon_{i}(z)\varepsilon^{*}_{i+1}(z):,&\quad
\mbox{if}~1\leqslant
i\leqslant m;\\
:\varepsilon_{m+1}(z)\delta^{*}_{1}(z):,&\quad\mbox{if}~i=m+1\\
:\delta_{i-m-1}(z)\delta^{*}_{i-m}(z):,&\quad \mbox{if}~m+2\leqslant
i\leqslant m+n+1.
\end{array}
\right.$$
$$x_{i}^{-}(z)=
\left\{\begin{array}{lll}
\sqrt{-1}:\beta(z)\delta_{n+1}^{*}(z):,&\quad\mbox{if}~ i=0;\\
\sqrt{-1}:\varepsilon_{i}(z)\varepsilon^{*}_{i+1}(z):,&\quad
\mbox{if}~1\leqslant
i\leqslant m;\\
:\varepsilon^{*}_{m+1}(z)\delta_{1}(z):,&\quad\mbox{if}~i=m+1\\
:\delta^{*}_{i-m-1}(z)\delta_{i-m}(z):,&\quad \mbox{if}~m+2\leqslant
i\leqslant m+n+1.
\end{array}
\right.$$
$$\alpha_{i}(z)=
\left\{\begin{array}{lll}
:\delta_{n+1}(z)\delta^{*}_{n+1}(z):-:\beta(z)\beta^{*}(z):,&\quad\mbox{if}~ i=0;\\
:\varepsilon_{i}(z)\varepsilon^{*}_{i}(z):-:\varepsilon_{i+1}(z)\varepsilon^{*}_{i+1}(z):,&\quad \mbox{if}~1\leqslant i\leqslant m;\\
:\varepsilon_{m+1}(z)\varepsilon^{*}_{m+1}(z):-:\delta_{1}(z)\delta^{*}_{1}(z):,&\quad\mbox{if}~i=m+1\\
:\delta_{i-m-1}(z)\delta^{*}_{i-m-1}(z):
   -:\delta_{i-m}(z)\delta^{*}_{i-m}(z):,&\quad
\mbox{if}~m+2\leqslant i\leqslant m+n+1.
\end{array}
\right.$$

(ii) For type $B(0,n)$, we define
$$x_{i}^{+}(z)=
\left\{\begin{array}{lll}
\frac{1}{2}:\beta^{*}(z)\beta^{*}(z):,&\quad\mbox{if}~ i=0;\\
:\varepsilon_{i}(z)\varepsilon^{*}_{i+1}(z):,&\quad
\mbox{if}~1\leqslant
i\leqslant n-1;\\
:\varepsilon_{n}(z)e(z):,&\quad \mbox{if}~\mbox{if}~i=n.
\end{array}
\right.$$
$$x_{i}^{-}(z)=
\left\{\begin{array}{lll}
\frac{1}{2}:\beta(z)\beta(z):,&\quad i=0;\\
-:\varepsilon_{i+1}(z)\varepsilon^{*}_{i}(z):,&\quad
\mbox{if}~1\leqslant
i\leqslant n-1;\\
:\varepsilon^{*}_{n}(z)e(z):,&\quad\mbox{if}~ i=n.
\end{array}
\right.$$
$$\alpha_{i}(z)=
\left\{\begin{array}{lll}
-2:\beta(z)\beta^{*}(z):,&\quad\mbox{if}~ i=0;\\
:\varepsilon_{i}(z)\varepsilon^{*}_{i}(z):-:\varepsilon_{i+1}(z)\varepsilon^{*}_{i+1}(z):,&\quad
\mbox{if}~1\leqslant i\leqslant n-1;\\
:\varepsilon_{n}(z)\varepsilon^{*}_{n}(z):,&\quad \mbox{if}~i=n.
\end{array}
\right.$$

\bigskip

(iii) For type $B(m,n)~(m\geqslant 1)$, the
$x_{i}^{\pm}(z),\alpha_{i}(z) ~(0\leqslant i\leqslant n-1)$ is same
as those of type $B(0,n)$ and the others is defined as:
$$x_{i}^{+}(z)=
\left\{\begin{array}{lll}
:\varepsilon_{n}(z)\delta^{*}_{1}(z):,&\quad\mbox{if}~ i=n;\\
:\delta_{i-n}(z)\delta^{*}_{i-n+1}(z):,&\quad \mbox{if}~n+1\leqslant
i\leqslant n+m-1;\\
:\delta_{m}(z)e(z):,&\quad\mbox{if}~i=n+m.\\
\end{array}
\right.$$
$$x_{i}^{-}(z)=
\left\{\begin{array}{lll}
:\delta_{1}(z)\varepsilon^{*}_{n}(z):,&\quad\mbox{if}~ i=n;\\
:\delta^{*}_{i-n}(z)\delta_{i-n+1}(z):,&\quad\mbox{if}~ n+1\leqslant
i\leqslant n+m-1;\\
:\delta^{*}_{m}(z)e(z):,&\quad\mbox{if}~i=n+m.\\
\end{array}
\right.$$
$$\alpha_{i}(z)=
\left\{\begin{array}{lll}
:\varepsilon_{n}(z)\varepsilon^{*}_{n}(z):-:\delta_{1}(z)\delta^{*}_{1}(z):,&\quad\mbox{if}~
i=n;\\
:\delta_{i-n}(z)\delta^{*}_{i-n}(z):-:\delta_{i-n+1}(z)\delta^{*}_{i-n+1}(z):,&\quad\mbox{if}~
n+1\leqslant i\leqslant n+m-1;\\
:\delta_{m}(z)\delta^{*}_{m}(z):,&\quad \mbox{if}~i=n.
\end{array}
\right.$$

(iv) For type $D(m,n)~(m\geqslant 1)$, the
$x_{i}^{\pm}(z),\alpha_{i}(z) ~(0\leqslant i\leqslant m+n-1)$ is
same as those of type $B(m,n)$ and the only difference is the last
one:
\begin{eqnarray*}
 &&x_{m+n}^{+}(z)=:\delta_{m-1}(z)\delta_{m}(z):,
 \quad x_{m+n}^{-}(z)=:\delta^{*}_{m-1}(z)\delta^{*}_{m}(z):\\
 &&\alpha_{m+n}(z)= :\delta_{m-1}(z)\delta^{*}_{m-1}(z)+:\delta_{m}(z)\delta^{*}_{m}(z)
\end{eqnarray*}

(v) For type $C(n)$, we define
$$x_{i}^{+}(z)=
\left\{\begin{array}{lll}
:\beta^{*}(z)\varepsilon^{*}_{1}(z):,&\quad\mbox{if}~ i=0;\\
:\varepsilon_{1}(z)\delta^{*}_{1}(z):,&\quad \mbox{if}~i=1\\
\sqrt{-1}:\delta_{i-1}(z)\delta^{*}_{i}(z):,&\quad \mbox{if}~2\leqslant i\leqslant n;\\
\frac{1}{2}:\delta_{n}(z)\delta_{n}(z):,&\quad \mbox{if}~i=n+1.
\end{array}
\right.$$
$$x_{i}^{-}(z)=
\left\{\begin{array}{lll}
:\beta(z)\varepsilon_{1}:,&\quad\mbox{if}~ i=0;\\
:\delta_{1}(z)\varepsilon_{1}^{*}(z):,&\quad \mbox{if}~i=1\\
\sqrt{-1}:\delta_{i}(z)\delta^{*}_{i-1}(z):,&\quad \mbox{if}~2\leqslant i\leqslant n;\\
\frac{1}{2}:\delta^{*}_{n}(z)\delta^{*}_{n}(z):,&\quad
\mbox{if}~i=n+1.
\end{array}
\right.$$
$$\alpha_{i}(z)=
\left\{\begin{array}{lll}
:\varepsilon(z_{1})\varepsilon_{1}^{*}(z):-:\beta(z)\beta^{*}(z):,&\quad \mbox{if}~i=0;\\
-:\varepsilon_{1}(z)\varepsilon^{*}_{1}(z):-:\delta_{1}(z)\delta^{*}_{1}(z):,&\quad\mbox{if}~
i=1;\\
:\delta_{i-1}(z)\delta^{*}_{i-1}(z):-:\delta_{i}(z)\delta^{*}_{i}(z):,&\quad\mbox{if}~
2\leqslant i\leqslant n;\\
2:\delta_{n}(z)\delta^{*}_{n}(z):,&\quad\mbox{if}~ i=n+1.
\end{array}
\right.$$
\begin{theorem} The field operators defined above give rise to a
representation of level $-1$ of $\mathfrak{T}(X)$ for type $A(m,n)$,
$B(m,n)$ and $D(m,n)$, of level $1$ for type
$C(n)$respectively.\end{theorem}
\begin{proof} To prove the theorem, we need to check the
field operators defined above satisfy all the  relations listed in
proposition 3.4 for each case.

i) Type $A(m,n)$
\begin{eqnarray*}
   [x_{0}^{+}(z),x_{0}^{-}(w)]&=&-[:\delta_{n+1}(z)\beta^{*}(z):,:\beta(z)\delta_{n+1}^{*}(z):] \\
   &=&-\big(:\delta_{n+1}(z)\delta^{*}_{n+1}(z):-:\beta(z)\beta^{*}(z):\big)\delta(z-w)
   +\partial_{w}\delta(z-w)\\
   &=&-\big(\alpha_{0}(z)\delta(z-w)
   +(-1)\partial_{w}\delta(z-w)\cdot\big)
\end{eqnarray*}
\begin{eqnarray*}
   [\alpha_{0}(z),x_{0}^{\pm}(w)]=0=\pm(\alpha_{0},\alpha_{0})x_{0}^{\pm}(w)\delta(z-w)
\end{eqnarray*}
For $1\leqslant i\leqslant m$, we have
\begin{eqnarray*}
   [x_{i}^{+}(z),x_{i}^{-}(w)]&=&[:\varepsilon_{i}(z)\varepsilon^{*}_{i+1}(z):,
   :\varepsilon_{i+1}(z)\varepsilon^{*}_{i}(z):] \\
   &=&-\big(:\varepsilon_{i}(z)\varepsilon^{*}_{i}(z):-:\varepsilon_{i+1}(z)
   \varepsilon^{*}_{i+1}(z):\big)\delta(z-w)
   +\partial_{w}\delta(z-w)\\
   &=&-\frac{2}{(\alpha_{i},\alpha_{i})}\big(\alpha_{i}(z)\delta(z-w)
   +(-1)\partial_{w}\delta(z-w)\cdot\big)
\end{eqnarray*}
\begin{eqnarray*}
   [\alpha_{i}(z),x_{i}^{+}(w)]&=&\sqrt{-1}[:\varepsilon_{i}(z)\varepsilon^{*}_{i}(z):
   -:\varepsilon_{i+1}(z)\varepsilon^{*}_{i+1}(z):,:\varepsilon_{i}(z)\varepsilon^{*}_{i+1}(z):] \\
   &=&2\sqrt{-1}:\varepsilon_{i}(z)\varepsilon^{*}_{i+1}(z):\delta(z-w)\\
   &=&(\alpha_{i},\alpha_{i})x_{i}^{+}(w)\delta(z-w)
  \end{eqnarray*}
$$[\alpha_{i}(z),x_{i}^{-}(w)]=-(\alpha_{i},\alpha_{i})x_{i}^{-}(w)\delta(z-w)$$
\begin{eqnarray*}
   [x_{m+1}^{+}(z),x_{m+1}^{-}(w)]&=&[:\varepsilon_{m+1}(z)\delta^{*}_{1}(z):,
   :\varepsilon^{*}_{m+1}(z)\delta_{1}(z):] \\
   &=&-\big(:\varepsilon_{m+1}(z)\varepsilon^{*}_{m+1}(z):-:\delta_{1}(z)
   \delta^{*}_{1}(z):\big)\delta(z-w)
   +\partial_{w}\delta(z-w)\\
   &=&-\big(\alpha_{m+1}(z)\delta(z-w)
   +(-1)\partial_{w}\delta(z-w)\cdot\big)
\end{eqnarray*}
\begin{eqnarray*}
   [\alpha_{m+1}(z),x_{m+1}^{\pm}(w)]=0=\pm(\alpha_{m+1},\alpha_{m+1})x_{m+1}^{\pm}(w)\delta(z-w)
\end{eqnarray*}
For $m+2\leqslant i\leqslant m+n+1$, we have
\begin{eqnarray*}
   [x_{i}^{+}(z),x_{i}^{-}(w)]&=&[:\delta_{i-m-1}(z)\delta^{*}_{i-m}(z):,
   :\delta_{i-m}(z)\delta^{*}_{i-m-1}(z):] \\
   &=&-\big(:\delta_{i-m-1}(z)\delta^{*}_{i-m-1}(z):
   -:\delta_{i-m}(z)\delta^{*}_{i-m}(z):\big)\delta(z-w)
   +\partial_{w}\delta(z-w)\\
   &=&-\big(\alpha_{i}(z)\delta(z-w)
   +(-1)\partial_{w}\delta(z-w)\cdot\big)
\end{eqnarray*}
\begin{eqnarray*}
   [\alpha_{i}(z),x_{i}^{+}(w)]&=&[:\delta_{i-m-1}(z)\delta^{*}_{i-m-1}(z):
   -:\delta_{i-m}(z)\delta^{*}_{i-m}(z):,
   :\delta_{i-m-1}(z)\delta^{*}_{i-m}(z):] \\
   &=&-2:\delta_{i-m-1}(z)\delta^{*}_{i-m}(z):\delta(z-w)=
   (\alpha_{i},\alpha_{i})x_{i}^{+}(w)\delta(z-w)
  \end{eqnarray*}
$$[\alpha_{i}(z),x_{i}^{-}(w)]=-(\alpha_{i},\alpha_{i})x_{i}^{-}(w)\delta(z-w)$$
It is straightforward to check for all $i\neq j$ that
$[x_{i}^{+}(z),x_{j}^{-}(w)]=0$ and
$$[\alpha_{i}(z),\alpha_{j}(w)]=(\alpha_{i},\alpha_{j})\partial_{w}\delta(z-w)\cdot(-1)$$
for example, we have
 \begin{eqnarray*}
 [\alpha_{i}(z),\alpha_{i+1}(w)] &=&-[:\varepsilon_{i+1}(z)\varepsilon^{*}_{i+1}(z):,
 :\varepsilon_{i+1}(w)\varepsilon^{*}_{i+1}(w):]  \\
  &=&\partial_{w}\delta(z-w)=(\alpha_{i},\alpha_{j})\partial_{w}\delta(z-w)\cdot(-1)
  \end{eqnarray*}
 We preceed to check the serre relations
 \begin{eqnarray*}
    && [x_{0}^{+}(z_{1}),[x_{0}^{+}(z_{2}),x_{1}^{+}(w)]]\\
    &=&-\sqrt{-1}[:\delta_{n+1}(z_{1})\beta^{*}(z_{1}):,[:\delta_{n+1}(z_{2})\beta^{*}(z_{2}):,
    :\varepsilon_{1}(w)\varepsilon^{*}_{2}(w):]]\\
    &=&-\sqrt{-1}[:\delta_{n+1}(z_{1})\beta^{*}(z_{1}):,:\delta_{n+1}(w)
    \varepsilon^{*}_{2}(w):]\delta(z_{2}-w)\\
    &=&0
 \end{eqnarray*}
 all the others relations are proved similarly.

\bigskip

ii) The construction for type $B(0,n)$ is almost the same to that
presented in \cite{JX} and the only difference is that we let $(e,e)=-2$
while it is 1 in \cite{JX}.

\bigskip

iii) Since the field operators for type $B(m,n)~(m\geqslant1)$ is
the same as type $B(0,n)$ for $0\leqslant i\leqslant n-1$, we only
to check field operators for $n\leqslant i\leqslant n+m$ satify the
relations. First, we have by Prososition \ref{P:comm1} 4) that
\begin{eqnarray*}
[x^{+}_{n}(z),x^{-}_{n}(w)]&=&[:\varepsilon_{n}(z)\delta^{*}_{1}(z):,:\varepsilon^{*}_{n}(w)\delta_{1}(w):]\\
&=&-(:\varepsilon_{n}(z)\varepsilon^{*}_{n}(z):-:\delta_{1}(z)\delta^{*}_{1}(z):)\delta(z-w)+\partial_{w}\delta(z-w)\\
&=&-(\alpha_{n}(z)\delta(z-w)+(-1)\cdot\partial_{w}\delta(z-w))
\end{eqnarray*}
and by Propostion  \ref{P:comm1} 5) and 6)
\begin{eqnarray*}
   [\alpha_{n}(z),x^{+}_{n}(w)]&=& [:\varepsilon_{n}(z)\varepsilon^{*}_{n}(z):
   -:\delta_{1}(z)\delta^{*}_{1}(z):,:\varepsilon_{n}(w)\delta^{*}_{1}(w):]\\
   &=&0=(\alpha_{n},\alpha_{n}),x^{+}_{n}(w)\delta(z-w)
   \end{eqnarray*}
     $$[\alpha_{n}(z),x^{-}_{n}(w)]
     =0=-(\alpha_{n},\alpha_{n})x^{-}_{n}(w)\delta(z-w)$$
For $n+1\leqslant i\leqslant m+n-1$, we have by proposition  4.3 3)
that\begin{eqnarray*}
[x^{+}_{i}(z),x^{-}_{i}(w)]&=&[:\delta_{i-n}(z)\delta^{*}_{i-n+1}(z):,:\delta^{*}_{i-n}(w)\delta_{i-n+1}(w):]\\
&=&\big(:\delta_{i-n}(z)\delta^{*}_{i-n}(z):-:\delta_{i-n+1}(z)\delta^{*}_{i-n+1}(z):\big)\delta(z-w)-\partial_{w}\delta(z-w)\\
&=&-\frac{2}{(\alpha_{i},\alpha_{i})}\big(\alpha_{i}(z)\delta(z-w)+(-1)\cdot\partial_{w}\delta(z-w)\big)
\end{eqnarray*}
\begin{eqnarray*}
   [\alpha_{i}(z),x^{+}_{i}(w)]&=& [:\delta_{i-n}(z)\delta^{*}_{i-n}(z):-:\delta_{i-n+1}(z)\delta^{*}_{i-n+1}(z):
   ,:\delta_{i-n}(z)\delta^{*}_{i-n+1}(z):]\\
   &=&-2:\delta_{i-n+1}(z)\delta^{*}_{i-n+1}(z):\delta(z-w)
=(\alpha_{i},\alpha_{i})x^{+}_{i}(w)\delta(z-w)
   \end{eqnarray*}
   $$[\alpha_{i}(z),x^{+}_{i}(w)]=-(\alpha_{i},\alpha_{i})x^{-}_{i}(w)\delta(z-w)\qquad\qquad$$
\begin{eqnarray*}
[x^{+}_{n+m}(z),x^{-}_{n+m}(w)]&=&[:\delta_{m}(z)e(z):,:\delta^{*}_{m}(w)e(w):]\\
&=&2:\delta_{m}(z)\delta^{*}_{m}(z):\delta(z-w)-2\partial_{w}\delta(z-w)\\
&=&-\frac{2}{(\alpha_{n+m},\alpha_{n+m})}\Big(\alpha_{n+m}(z)\delta(z-w)+(-1)\cdot\partial_{w}\delta(z-w)\Big)
\end{eqnarray*}
where  we have used the fact if $u(z)$ is an odd field, then
$:u(z)u(z):=0$. \begin{eqnarray*}
   [\alpha_{n+m}(z),x^{+}_{n+m}(w)]&=& [:\delta_{m}(z)\delta^{*}_{m}(z):
   ,:\delta_{m}(z)e(z):]\\
   &=&-:\delta_{m}(z)e(z):\delta(z-w)
=(\alpha_{n+m},\alpha_{n+m})x^{+}_{n+m}(w)\delta(z-w)
   \end{eqnarray*}
$$[\alpha_{n+m}(z),x^{-}_{n+m}(w)]=-(\alpha_{n+m},\alpha_{n+m})x^{-}_{n+m}(w)\delta(z-w)$$

One can check that $[x^{+}_{i}(z),x^{-}_{j}(w)]=0$ for $i\neq j$ one
by one using Proposition \ref{P:comm1}.

In what follows,  we check those relations concerning on
$\alpha_{i}(z)~(n\leqslant i\leqslant n+m)$. First, one has
\begin{eqnarray*}
  [\alpha_{n-1}(z),\alpha_{n}(w)]
  &=&-[:\varepsilon_{n}(z)\varepsilon^{*}_{n}(z):
  ,:\varepsilon_{n}(w)\varepsilon^{*}_{n}(w):] \\
   &=&\partial_{w}\delta(z-w)=(\alpha_{n-1},\alpha_{n})\partial_{w}\delta(z-w)\cdot(-1)
\end{eqnarray*}
Then we have, for $n\leqslant i\leqslant n+m-1$, that
\begin{eqnarray*}
  [\alpha_{i}(z),\alpha_{i+1}(w)]
  &=&-[:\delta_{i-n+1}(z)\delta^{*}_{i-n+1}(z):
  ,:\delta_{i-n+1}(w)\delta^{*}_{i-n+1}(w):] \\
   &=&-\partial_{w}\delta(z-w)=(\alpha_{i},\alpha_{i+1})\partial_{w}\delta(z-w)\cdot(-1)
\end{eqnarray*}

Next, we proceed to check the Serre type relations. First of all, we
have $[x_{i}^{\pm}(z),x_{i}^{\pm}(w)]=0$ for $0\leqslant i\leqslant
m+n$ and $[x_{n}^{\pm}(z),x_{i}^{\pm}(w)]=0$ for $0\leqslant
i\leqslant m+n,i\neq n-1,n+1$. Then we check that
\begin{eqnarray*}
  &&[x_{n}^{+}(z_{1}),[x_{n}^{+}(z_{1}),x_{n-1}^{+}(w)]]\\
  &&\qquad=-[:\varepsilon_{n}(z_{1})\delta^{*}_{1}(z_{1}):,:\delta^{*}_{1}(w)\varepsilon_{n}(w):]\delta(z_{2}-w)=0
\end{eqnarray*}
\begin{eqnarray*}
  &&[x_{n-1}^{+}(z_{1}),[x_{n-1}^{+}(z_{1}),x_{n}^{+}(w)]]\\
  &&\qquad=-[:\varepsilon_{n-1}(z_{1})\delta^{*}_{n}(z_{1}):,:\varepsilon_{n-1}(w)\delta^{*}_{1}(w):]\delta(z_{2}-w)=0
\end{eqnarray*}
similary, we have for $n\leqslant i\leqslant n+m-2$
$$[x_{i}^{\pm}(z_{1}),[x_{i}^{\pm}(z_{1}),x_{i+1}^{\pm}(w)]]
=[x_{i+1}^{\pm}(z_{1}),[x_{i+1}^{\pm}(z_{1}),x_{i}^{\pm}(w)]]=0$$
Finally, we check that
\begin{eqnarray*}
&& [x_{m+n}^{+}(z_{1}),[x_{m+n}^{+}(z_{2}),[x_{m+n}^{}(z_{3}),x_{m+n-1}^{+}(w)]]]\\
&=&-[:\delta_{m}(z_{1})e(z_{1}):,[:\delta_{m}(z_{2})e(z_{2}):,:e(w)\delta_{m+n-1}(w):]\delta(z_{3}-w)\\
&=&-2[:\delta_{m}(z_{1})e(z_{1}):,:\delta_{m}(w)\delta_{m-1}(w):]\delta(z_{2}-w)\delta(z_{3}-w)=0
\end{eqnarray*}
and $[x_{m+n-1}^{+}(z_{1}),[x_{m+n-1}^{+}(z_{2})x_{m+n}^{+}(w)]]=0$.

\bigskip

iv) For type $D(m,n)$, note that the difference betwwen $B(m,n)$,
then it is sufficient to  check that
\begin{eqnarray*}
[x^{+}_{n+m}(z),x^{-}_{n+m}(w)]&=&[:\delta_{m-1}(z)\delta_{m}(z):,:\delta^{*}_{m-1}(w)\delta^{*}_{m}(w):]\\
&=&\big(:\delta_{m-1}(z)\delta^{*}_{m-1}(z):+:\delta_{m}(z)\delta^{*}_{m}(z):\big)\delta(z-w)-\partial_{w}\delta(z-w)\\
&=&-\frac{2}{(\alpha_{n+m},\alpha_{n+m})}\Big(\alpha_{n+m}(z)\delta(z-w)+(-1)\cdot\partial_{w}\delta(z-w)\Big)
\end{eqnarray*}
where  we have used the fact if $u(z)$ is an odd field, then
$:u(z)u(z):=0$. \begin{eqnarray*}
   [\alpha_{n+m}(z),x^{+}_{n+m}(w)]&=& [:\delta_{m-1}(z)\delta^{*}_{m-1}(z):+:\delta_{m}(z)\delta^{*}_{m}(z):
   ,:\delta_{m-1}(z)\delta_{m}(z):]\\
   &=&-2:\delta_{m-1}(z)\delta_{m}(z):\delta(z-w)
=(\alpha_{n+m},\alpha_{n+m})x^{+}_{n+m}(w)\delta(z-w)
   \end{eqnarray*}
$$[\alpha_{n+m}(z),x^{-}_{n+m}(w)]=-(\alpha_{n+m},\alpha_{n+m})x^{-}_{n+m}(w)\delta(z-w)$$
$$ [x^{\pm}_{n+m}(z),x^{\mp}_{i}(w)]=0,~0\leqslant i\leqslant n+m-1  $$
\begin{eqnarray*}
  [\alpha_{n+m-2}(z),\alpha_{n+m}(w)]
  &=&-[:\delta_{m-1}(z)\delta^{*}_{m-1}(z):
  ,:\delta_{m-1}(w)\delta^{*}_{m-1}(w):] \\
   &=&\partial_{w}\delta(z-w)=(\alpha_{n+m-2},\alpha_{n+m})\partial_{w}\delta(z-w)\cdot(-1)
\end{eqnarray*}
$$[\alpha_{i}(z),\alpha_{n+m}(w)]=0,i\neq n+m-2$$
The Serre relations need checking here are those concerning
$x_{m+n}^{\pm}(z)$ and $x_{m+n-2}^{\pm}(z)$, for example
\begin{eqnarray*}
  &&[x_{m+n}^{+}(z_{1}),[x_{m+n}^{+}(z_{2}),x_{m+n-2}^{+}(w)]\\
  &&\qquad\qquad=-[:\delta_{m-1}(z_{1})\delta_{m}(z_{1}):,:\delta_{m}(w)\delta_{m-2}(w):]\delta(z_{2}-w)=0
\end{eqnarray*}
all others can be proved similarly.

\bigskip

v) For type $C(n)$, we first check relations 3') and 4') in
Prososition \ref{P:comm1}
\begin{eqnarray*}
[x_{0}^{+}(z),x_{0}^{-}(w)]&=&[:\beta^{*}(z)\varepsilon^{*}_{1}(z):,:\beta(z)\varepsilon_{1}(z):]\\
&=&-\Big(\big(:\varepsilon(z_{1})\varepsilon_{1}^{*}(z):-:\beta(z)\beta^{*}(z):\big)
\delta(z-w)+1\cdot\partial_{w}\delta(z-w)\Big)\\
&=&-(\alpha_{0}(z)\delta(z-w)+1\cdot\partial_{w}\delta(z-w))
\end{eqnarray*}
\begin{eqnarray*}
[\alpha_{0}(z),x_{0}^{\pm}(w)]=0=\pm(\alpha_{0},\alpha_{0})x_{0}^{\pm}(w)\delta(z-w)
\end{eqnarray*}

\begin{eqnarray*}
[x_{1}^{+}(z),x_{1}^{-}(w)]&=&[:\varepsilon_{1}(z)\delta^{*}_{1}(z):,:\varepsilon_{1}^{*}(z)\delta_{1}(z):]\\
&=&-\Big(\big(:\varepsilon_{1}(z)\varepsilon^{*}_{1}(z):-:\delta_{1}(z)\delta^{*}_{1}(z):\big)
\delta(z-w)+1\cdot\partial_{w}\delta(z-w)\Big)\\
&=&-(\alpha_{1}(z)\delta(z-w)+1\cdot\partial_{w}\delta(z-w))
\end{eqnarray*}
\begin{eqnarray*}
[\alpha_{1}(z),x_{1}^{\pm}(w)]=0=\pm(\alpha_{1},\alpha_{1})x_{1}^{\pm}(w)\delta(z-w)
\end{eqnarray*}
\begin{eqnarray*}
[x_{i}^{+}(z),x_{i}^{-}(w)]&=&-[:\delta_{i-1}(z)\delta^{*}_{i}(z):,:\delta_{i}(z)\delta^{*}_{i-1}(z):]\\
&=&\big(:\delta_{i-1}(z)\delta^{*}_{i-1}(z):-:\delta_{i}(z)\delta^{*}_{i}(z):\big)
\delta(z-w)+1\cdot\partial_{w}\delta(z-w)\big)\\
&=&-\frac{2}{(\alpha_{i},\alpha_{i})}\big(\alpha_{i}(z)\delta(z-w)+1\cdot\partial_{w}\delta(z-w)\big),
\quad 2\leqslant i\leqslant n.
\end{eqnarray*}
\begin{eqnarray*}
[\alpha_{i}(z),x_{i}^{+}(w)]&=&\sqrt{-1}[:\delta_{i-1}(z)\delta^{*}_{i-1}(z):
-:\delta_{i}(z)\delta^{*}_{i}(z):,:\delta_{i-1}(z)\delta^{*}_{i}(z):]\\
&=&-2\sqrt{-1}:\delta_{i-1}(z)\delta^{*}_{i}(z):\delta(z-w)\\
&=&(\alpha_{i},\alpha_{i})x_{i}^{+}(w)\delta(z-w)
\end{eqnarray*}
$$[\alpha_{i}(z),x_{i}^{-}(w)]=-(\alpha_{i},\alpha_{i})x_{i}^{-}(w)\delta(z-w)$$

\begin{eqnarray*}
[x_{n+1}^{+}(z),x_{n+1}^{-}(w)]&=&\frac{1}{4}[:\delta_{n}(z)\delta_{n}(z):,
:\delta^{*}_{n}(z)\delta^{*}_{n}(z):]\\
&=&\frac{2}{4}\big(2:\delta_{n}(z)\delta^{*}_{n}(z):
\delta(z-w)+1\cdot\partial_{w}\delta(z-w)\big)\\
&=&-\frac{2}{(\alpha_{n+1},\alpha_{n+1})}\big(\alpha_{n+1}(z)\delta(z-w)+1\cdot\partial_{w}\delta(z-w)\big)
\end{eqnarray*}
\begin{eqnarray*}
[\alpha_{n+1}(z),x_{n+1}^{+}(w)]&=&[:\delta_{n}(z)\delta^{*}_{n}(z):
,:\delta_{n}(z)\delta_{n}(z):]\\
&=&-2:\delta_{i}(z)\delta^{*}_{i+1}(z):\delta(z-w)
=(\alpha_{n+1},\alpha_{n+1})x_{n+1}^{+}(w)\delta(z-w)
\end{eqnarray*}
$$[\alpha_{n+1}(z),x_{n+1}^{-}(w)]=-(\alpha_{n+1},\alpha_{n+1})x_{n+1}^{-}(w)\delta(z-w)$$
$$[x_{i}^{+}(z),x_{j}^{-}(w)]=0,\quad\mbox{for}~i\neq j$$

Next, we will check the relation 3') by straightforward calculus
$$[\alpha_{0}(z),\alpha_{0}(w)]=[\alpha_{1}(z),\alpha_{1}(w)]=0$$
\begin{eqnarray*}
[\alpha_{i}(z),\alpha_{i}(w)&=&[:\delta_{i-1}(z)\delta^{*}_{i-1}(z):,:\delta_{i-1}(w)\delta^{*}_{i-1}(w):]\\
&&+[:\delta_{i-1}(z)\delta^{*}_{i-1}(z):,:\delta_{i-1}(w)\delta^{*}_{i-1}(w):],\qquad 2\leqslant i\leqslant n\\
&=&-2\partial_{w}\delta(z-w)=(\alpha_{i},\alpha_{i})\partial_{w}\delta(z-w)
\end{eqnarray*}
\begin{eqnarray*}
[\alpha_{n+1}(z),\alpha_{n+1}(w)&=&4[:\delta_{n}(z)\delta^{*}_{n}(z):,:\delta_{n}(w)\delta^{*}_{n}(w):]\\
&=&-4\partial_{w}\delta(z-w)=(\alpha_{n+1},\alpha_{n+1})\partial_{w}\delta(z-w)
\end{eqnarray*}
all others are proved similarly.

Finally, we proceed to check the serre type relations as follows:
$$[x_{i}^{\pm}(z),x_{i}^{\pm}(w)]=0,\quad 0\leqslant i\leqslant n+1$$
$$[x_{0}^{\pm}(z),x_{i}^{\pm}(w)]=[x_{1}^{\pm}(z),x_{i}^{\pm}(w)]=0,\quad 3\leqslant i\leqslant n+1$$
\begin{eqnarray*}
  &&[x_{0}^{+}(z_{1}),[x_{0}^{+}(z_{2}),x_{1}^{+}(z_{w})]]=
 [:\beta^{*}(z_{1})\varepsilon^{*}_{1}(z_{1}):,:\beta^{*}(w)\varepsilon^{*}_{1}(w):]\delta(z_{2}-w)=0
\end{eqnarray*}
\begin{eqnarray*}
  &&[x_{n}^{+}(z_{1}),[x_{n}^{+}(z_{2}),[x_{n}^{+}(z_{3}),x_{n+1}^{+}(z_{w})]]\\
  =&&\sqrt{-1}[:\delta_{n-1}(z_{1})\delta^{*}_{n}(z_{1}):,
  [:\delta_{n-1}(z_{2})\delta^{*}_{n}(z_{2}):,:\delta_{n-1}(w)\delta_{n}(w):]\delta(z_{3}-w)\\
  =&&-\sqrt{-1}[:\delta_{n-1}(z_{1})\delta^{*}_{n}(z_{1}):,:\delta_{n-1}(w)\delta_{n-1}(w):]\delta(z_{2}-w)\delta(z_{3}-w)\\
  =&&0
\end{eqnarray*}
all others are proved similarly and this completes the proof.
\end{proof}

\section{Appendix}
In this appendix, we list the extended distinguished Cartan matrix
of affine superalgebra  of type $A,B,C, D$ for convenience.

1) Type $A(m,n)$
$$
\begin{pmatrix}
 2 & -1& 0 &\cdots  &0&0&0&\cdots& &1 \\
-1 & 2 & -1& \ddots &0&0&0&\cdots& &0 \\
 0 &-1 & 2 & \ddots &0& &0&\cdots& &0 \\
\vdots&\ddots&\ddots&\ddots&-1&\ddots&&&&\vdots\\
0&&0&-1&2&-1&0&&&0\\
0 &\cdots & 0 & 0 & -1 & 0 &1&\cdots&&0 \\
0 &\cdots & 0 & 0 & 0 & -1 &2&\cdots&&0 \\
\vdots&\ddots&\ddots&\ddots&\ddots&\ddots&\ddots&&&\vdots\\
0&  & &  &&&& 2&-1 & 0 \\
0 &  &&  &&&& -1 &2 & -1 \\
-1 & \cdots &&\cdots&&\cdots&&0&-1 & 2
\end{pmatrix}$$

2) type $B(0,n)^{(1)}~(n\geqslant1)$
\begin{equation*}
\begin{pmatrix} 2 & -1 & 0 &\cdots & 0& 0 & 0 \\
-2 & 2 & -1& \cdots & 0& 0 & 0 \\
0 & -1 & 2& \cdots & 0& 0 & 0 \\
\cdot & \cdot & \cdot & \cdots & \cdot & \cdot &\cdot \\
0 & 0 & 0& \cdots & 2&-1 & 0 \\
0 & 0 & 0& \cdots & -1 &2 & -1 \\
0 & 0 & 0& \cdots &0&-2 & 2
\end{pmatrix},
\end{equation*}
3) type $B(m,n)^{(1)}~(m\geqslant1,n\geqslant1)$
$$
A=\begin{pmatrix}
 2 & -1& 0 &\cdots  &0&0&0&\cdots& &0 \\
-2 & 2 & -1& \ddots &0&0&0&\cdots& &0 \\
 0 &-1 & 2 & \ddots &0& &0&\cdots& &0 \\
\vdots&\ddots&\ddots&\ddots&-1&\ddots&&&&\vdots\\
0&&0&-1&2&-1&0&&&0\\
0 &\cdots & 0 & 0 & -1 & 0 &1&\cdots&&0 \\
0 &\cdots & 0 & 0 & 0 & -1 &2&\cdots&&0 \\
\vdots&\ddots&\ddots&\ddots&\ddots&\ddots&\ddots&&&\vdots\\
0&  & &  &&&& 2&-1 & 0 \\
0 &  &&  &&&& -1 &2 & -1 \\
0 & \cdots &&\cdots&&\cdots&&0&-2 & 2
\end{pmatrix}$$

4) type $C(n)^{(1)}~(n\geqslant1)$
$$
\begin{pmatrix}

 0  & -2 &-1&\cdots&&&0 \\
 -2 & 0 &-1&\cdots&&&0 \\
 -1 & -1&2&\cdots&&&0 \\
&\ddots&\ddots&\ddots&&&&\vdots\\
0  &&&& 2&-1 & 0 \\
  &&&& -1 &2 & -2 \\
0&\cdots&&&0&-1 & 2
\end{pmatrix}$$

5) type $D(m,n)^{(1)}~(n\geqslant1)$
$$
\begin{pmatrix}
 2 & -1& 0 &\cdots  &0&0&0&\cdots& &0 \\
-2 & 2 & -1& \ddots &0&0&0&\cdots& &0 \\
 0 &-1 & 2 & \ddots &0& &0&\cdots& &0 \\
\vdots&\ddots&\ddots&\ddots&-1&\ddots&&&&\vdots\\
0&&0&-1&2&-1&0&&&0\\
0 &\cdots & 0 & 0 & -1 & 0 &1&\cdots&&0 \\
0 &\cdots & 0 & 0 & 0 & -1 &2&\cdots&&0 \\
\vdots&\ddots&\ddots&\ddots&\ddots&\ddots&\ddots&&&\vdots\\
0&  & &  &&&& 2&-1 &  -1 \\
0 &  &&  &&&& -1 &2 & 0\\
0 & &&\cdots&&\cdots&& -1&0 & 2
\end{pmatrix},
$$

\end{document}